\documentclass{amsart}

\usepackage{amsmath,amssymb,amsfonts,amstext,amsthm}
\usepackage{url}
\usepackage{graphicx,caption,subcaption}
\input xy
\xyoption{all}
\usepackage[all]{xy}
\usepackage{mathrsfs}

\graphicspath{{images/}{../images/}}
\usepackage{tikz}

\newcommand{\nwc}{\newcommand}

\nwc{\aaa}{\mathcal{F}}
\nwc{\aap}{\mathcal{F}_{P}}
\nwc{\al}{\alpha}

\nwc{\C}{\mathbb{C}}
\nwc{\cb}{\overline{C}}
\nwc{\ccc}{\mathfrak{c}}
\nwc{\ch}{\widehat{C}}
\nwc{\cin}{\textbf{(v)}}
\nwc{\cl}{C'}
\nwc{\cp}{\mathcal{C}_{P}}
\nwc{\cpll}{\mathfrak{c}_{P'}}
\nwc{\ct}{\widetilde{C}}

\nwc{\dd}{\mathcal{L}}
\nwc{\ddd}{\mathfrak{d}}
\nwc{\ddl}{\mathcal{L}'}
\nwc{\dlp}{\delta_{P}}
\nwc{\doi}{\textbf{(ii)}}

\nwc{\enq}{$$}

\nwc{\fl}{\flushleft}
\nwc{\fff}{\mathcal{F}}
\nwc{\ffp}{\mathcal{F}_{P}}
\nwc{\ffq}{\mathcal{F}_{Q}}
\nwc{\ffl}{\mathcal{F}'}

\nwc{\G}{\mathcal{G}}
\nwc{\Ga}{\Gamma}
\nwc{\gtl}{\widetilde{g}}

\nwc{\hra}{\hookrightarrow}
\nwc{\hua}{h^{1}(C,\aaa )}

\nwc{\kk}{{\rm K}}

\nwc{\llb}{\mathcal{L}}

\nwc{\mb}{\mathbb}
\nwc{\mc}{\mathcal}
\nwc{\mm}{\mathfrak{m}}
\nwc{\mmp}{\mathfrak{m}_{P}}
\nwc{\mpd}{\mathfrak{m}_{P}^{2}}

\nwc{\nn}{\mathbb{N}}

\nwc{\ob}{\overline{\mathcal{O}}}
\nwc{\obr}{\mathcal{O}^*}
\nwc{\obp}{\overline{\mathcal{O}}_P}
\nwc{\och}{\mathcal{O}_{\hat{C}}}
\nwc{\oh}{\hat{\mathcal{O}}}
\nwc{\ohp}{\hat{\mathcal{O}}_{P}}
\nwc{\ol}{\mathcal{O}'}
\nwc{\oma}{\Omega (\mathfrak{a})}
\nwc{\omo}{\Omega (\mathcal{O})}
\nwc{\oo}{\mathcal{O}}
\nwc{\op}{\mathcal{O}_P}
\nwc{\opc}{\mathcal{O}_{P,C}}
\nwc{\oph}{\hat{\mathcal{O}}_{P}}
\nwc{\opl}{\mathcal{O}_{P}'}
\nwc{\oplc}{\mathcal{O}_{P,C}'}
\nwc{\opll}{\mathcal{O}_{P'}}
\nwc{\opt}{\tilde{\mathcal{O}}_{P}}
\nwc{\optt}{{\mathcal{O}}_{\tilde{P}}}
\nwc{\oq}{\mathcal{O}_{Q}}
\nwc{\oqt}{\tilde{\mathcal{O}}_{Q}}
\nwc{\ot}{\widetilde{\mathcal{O}}}
\nwc{\overop}{\bar{\oo}_{P}}

\nwc{\pb}{\overline{P}}
\nwc{\pbb}{P^*}
\nwc{\pbi}{\overline{P_{i}}}
\nwc{\pbr}{\overline{P_{r}}}
\nwc{\pgmd}{\mathbb{P}^{g+2}}
\nwc{\pgmu}{\mathbb{P}^{g+1}}
\nwc{\ph}{\hat{P}}
\nwc{\pp}{\mathbb{P}}
\nwc{\prv}{\noindent\textbook{Proof}:}
\nwc{\pt}{\widetilde{P}}
\nwc{\ptl}{\tilde{P}}
\nwc{\pum}{\mathbb{P}^{1}}

\nwc{\qh}{\hat{Q}}
\nwc{\qtl}{\tilde{Q}}
\nwc{\qua}{\textbf{(iv)}}

\nwc{\ra}{\rightarrow}
\nwc{\rh}{\hat{R}}

\nwc{\sei}{\textbf{(vi)}}
\nwc{\sep}{\beq\ast\ \ast\ \ast\enq}
\nwc{\sig}{\sigma}
\nwc{\Sig}{\Sigma}
\nwc{\ssp}{S_{P}}
\nwc{\sss}{{\rm S}}

\nwc{\tre}{\textbf{(iii)}}

\nwc{\um}{\textbf{(i)}}

\nwc{\vpb}{v_{\overline{P}}}
\nwc{\vtxp}{\widetilde{V}_{x,P}}
\nwc{\vxp}{V_{x,P}}

\let \wt=\widetilde
\let \La=\Lambda

\nwc{\wh}{\hat{\omega}}
\nwc{\whp}{\hat{\omega}_{P}}
\nwc{\woch}{\omega\cdot\mathcal{O}_{\hat{C}}}
\nwc{\woh}{\omega\cdot\hat{\mathcal{O}}}
\nwc{\ww}{\omega}
\nwc{\wwb}{\omega^*}
\nwc{\wwct}{\omega _{\widetilde{C}}}
\nwc{\wwh}{\widehat{\omega}}
\nwc{\wwhp}{\widehat{\omega}_P}
\nwc{\wwp}{\omega _{P}}
\nwc{\wwt}{\widetilde{\omega}}
\nwc{\wwtp}{\widetilde{\omega}_P}

\nwc{\zz}{\mathbb{Z}}

\newtheorem{coro}{Corollary}[section]

\newtheorem{lemma}[coro]{Lemma}

\newtheorem{thm}[coro]{Theorem}
\newtheorem{conj}[coro]{Conjecture}

\let \fl=\flushleft
\let \fr=\frac

\let \ga=\gamma

\let \be=\beta
\let \al=\alpha
\let \pr=\prime
\let \la=\lambda

\let \ov=\overline
\let \om=\omega

\let \Ga=\Gamma
\let \De=\Delta

\begin{document}

\title{Real inflection points of real linear series on an elliptic curve}
\date{\empty}

\author{Ethan Cotterill}

\address{Instituto de Matem\'atica, UFF, Rua Prof Waldemar de Freitas, S/N,
24.210-201 Niter\'oi RJ, Brazil}

\email{cotterill.ethan@gmail.com}

\author{Cristhian Garay L\'opez}

\address{Departamento de Matem\'aticas, Centro de Investigaci\'on y de Estudios Avanzados del IPN, Apartado Postal 14-740, 07000. Ciudad de M\'exico, M\'exico.}

\email{cgaray@math.cinvestav.mx}

\maketitle

\begin{abstract}
Given a real elliptic  curve $E$ with non-empty real part and $[D]\in \mbox{Pic}^2 E$ its $g_2^1$, we study the real inflection points of distinguished subseries of the complete real linear series $|\mathcal{L}_\mathbb{R}(kD)|$ for $k\geq 3$. We define {\it key polynomials} whose roots index the ($x$-coordinates of) inflection points of the linear series, away from the points where $E$ ramifies over $\mathbb{P}^1$. These fit into a recursive hierarchy, in the same way that division polynomials index torsion points.

Our study is motivated by, and complements, an analysis of how inflectionary loci vary in the degeneration of real {\it hyperelliptic} curves to a metrized complex of curves with elliptic curve components that we carried out in \cite{BCG}.
\end{abstract}

\section{Introduction}
This note is a companion to the papers \cite{Ga} and \cite{BCG}. In \cite{Ga}, the second author described the topology of real inflectionary loci of complete linear series on a real elliptic curve $E$ as a function of the topology of the real locus $E(\mb{R})$, and used this as the inductive basis for a construction of real canonical curves of genus four in $\mb{P}^3$ with many real Weierstrass points. Then in \cite{BCG}, joint with I. Biswas, we used a degeneration borrowed from non-Archimedean geometry to relate two distinct generalizations of real inflectionary loci of complete series on $E$, namely
\begin{enumerate}
\item real inflectionary loci of complete series on real {\it hyper}elliptic curves $X$ of genus $g\geq2$; and
\item real inflectionary loci of {\it in}complete series on real elliptic curves.
\end{enumerate}
Very roughly, generalization (1) may be degenerated to generalization (2); a bit more precisely, the degeneration outputs a limit linear series on a metrized complex of real curves, whose models are smooth elliptic real curves $E_i$, $i=1,\dots,g$.

\medskip
Even more precisely, the complete real series $|\mc{L}_{\mathbb{R}}|$ studied in \cite{BCG} are all multiples of the $g^1_2$ obtained from pulling back a distinguished point $\infty$ on $\mb{P}^1$ via the hyperelliptic structure morphism $\pi: X \ra \mb{P}^1$. With respect to a suitable choice of affine coordinates, we can realize $X$ as (a projective completion of) $y^2=f(x)$, where $\deg(f)=2g+1$ and the map $\pi$ is given by $(x,y) \mapsto x$. Assume that $\mc{L}_{\mathbb{R}}$ is represented by the divisor $kD$ on $X$, where $D$ is the pullback of $\infty$ on $\mb{P}^1$. The interesting cases are associated with choices for which $k>g\geq2$, so hereafter we assume this. There is a distinguished real basis of holomorphic sections for $|\mc{L}_{\mathbb{R}}|= |\mc{L}_{\mathbb{R}}(kD)|$ given by
\begin{equation}\label{distinguished_basis}
\mc{F}_g= \{1,x,\dots,x^k; y, yx, \dots, yx^{k-g-1}\}.
\end{equation}

\medskip
The ramification locus $R_{\pi}$ of $\pi$ is comprised of real points that degenerate to the ramification points of the elliptic components $E_i$, viewed as double covers of $\mb{P}^1$. Furthermore, the $2k-g+1$ sections  $\mc{F}_g$ degenerate to $2k-g+1$ sections that we will also abusively denote by $\mc{F}_g$; these, in turn, determine an inflectionary basis in every point of the ramification locus of $E_i \ra \mb{P}^1$. Hereafter, we assume $i=1$. Abusively, we let $\pi$ denote the double cover $E=E_1 \ra \mb{P}^1$. In \cite{BCG}, we explicitly computed the contribution to (real) inflection made by each point of $R_{\pi}$. Note that in general the inflection arising from $R_{\pi}$ is not simple, which contrasts with the behavior of inflectionary loci of complete series. The behavior of inflectionary loci (including local multiplicities) away from $R_{\pi}$, however, was unclear. Here we will address this mystery behavior through a more careful analysis of the Wronskians whose zero loci select for (the $x$-coordinates of) inflectionary loci. 

\medskip
One upshot of our analysis will be that inflection is in fact simple (i.e. of multiplicity one) away from $R_{\pi}$, at least when the topology of $E(\mb{R})$ is maximally real. One of the key features of the complete series case on which the analysis of \cite{Ga} is predicated is that inflection points of a complete linear series on an elliptic curve correspond to torsion points. Torsion points, in turn, form a sub-lattice of the lattice $\La$ for which $E(\mb{C}) \cong \mb{C}/\La$ as a topological space. Their $x$-coordinates are computed by the {\it division polynomials} described in \cite{S}. By analogy, we will define {\it (generalized) key polynomials} that compute the $x$-coordinates of inflection points of $\mc{F}_g$ on $E$ away from $R_{\pi}$. Our key polynomials are extracted from Wronskians, which also play an important r\^ole in Griffiths and Harris' modern treatment \cite{GH} of Poncelet's theorem on polygons inscribed (and circumscribed) in conics. However, we organize the information encoded by these determinants in an apparently novel way. When $E(\mb{R})$ is maximally real, we can realize $E=E(\la)$ as an elliptic curve in Legendre form, where $\la$ is a real parameter. The inflectionary loci of curves in the associated flat family then determine a plane curve $\mc{C}=\mc{C}(k,g)$, whose singularities control in a rather precise way the (real) inflectionary loci of the curves $E(\la)$. Hopefully we will manage to convince the reader that the further study of key polynomials and curves is interesting in its own right.

\medskip
The authors would like to acknowledge the valuable help of Eduardo Ruiz-Duarte.

\section{Wronskians, revisited}
Borrowing notationally from \cite{BCG}, assume that $E$ is defined by an affine equation
$y^2= f$, where $f\in\mb{R}[x]$ is a separable cubic.
The restriction of the real inflectionary locus  on $E$ to the affine chart $U_y$ where the coordinate $y$ is nonvanishing is the zero locus of (the restriction of) a regular function $\al$ equal to a Wronskian determinant of partial derivatives of sections of $\mc{F}$, namely
\[
\al|_{U_y}= \det(f_i^{(j)})_{0 \leq i,j \leq 2k-g}.
\]
Here $f_i$ refers to the $i$-th element of the distinguished basis \eqref{distinguished_basis} of $\mc{F}_g$, and the superscript denotes the (order of) differentiation with respect to $x$.

\medskip
Now, for any given $n \geq 1$, define $P_n(x) \in \mb{R}[x]$ by
\begin{equation}\label{P_n}
y^{(n)}= f^{-n}y P_n(x).
\end{equation}

\medskip
Note that $\al|_{U_y}=y^{(n)}$ precisely when $n=k+1\geq3$ and $g=k-1$. It follows in this situation that the roots $x=\ga$ of $P_n(x)$ for which $f(\ga)>0$ are the $x$-coordinates of the inflection points of the codimension-$(g-1)$ subseries of $|\mc{L}_{\mathbb{R}}|$ spanned by $\mc{F}_g$
away from the ramification locus $R_{\pi}$. 
(Note that the positivity of $f$ is required in order for the root of a generalized key polynomial to lift to an inflection point.) 

\medskip
Accordingly, here we will focus on the calculation of the {\it key polynomials} $P_n(x)$, which may be done recursively. Differentiating \eqref{P_n} yields
\[
\begin{split}
y^{(n+1)}&= \frac{d P_n}{dx} \cdot f^{-n}y+ P_n(-nf^{-(n+1)}f^{\pr}y+ f^{-n}y^{\pr}) \\
&= \frac{d P_n}{dx} \cdot f^{-n}y+ P_n(-nf^{-(n+1)}f^{\pr}y+ f^{-n} \cdot \frac{1}{2} f^{\pr}f^{-1}y) \\
&= f^{-(n+1)}y \cdot \bigg(\frac{d P_n}{dx} \cdot f+ P_nf^{\pr}(-n+ 1/2) \bigg)
\end{split}
\]
from which we deduce that
\begin{equation}\label{P_n_recursion}
P_{n+1}= \frac{d P_n}{dx} \cdot f+ (-n+ 1/2) P_n \cdot \frac{df}{dx}.
\end{equation}
The key polynomials $P_n$, $n \geq 1$, are determined by the recursion \eqref{P_n_recursion} together with the seed datum
\begin{equation}\label{seed_datum}
P_1= \frac{1}{2} \frac{df}{dx}.
\end{equation}

Empirically we observe the following phenomenon.

\begin{conj} \label{separability_conjecture} When the polynomial $f$ has three distinct real roots, the $n$th key polynomial $P_n$ is separable, for all $n \geq 1$.
\end{conj}

\medskip 
Indeed, the assumption that $f$ has three distinct real roots means that the topology of $E(\mb{R})$ is maximally real, and that $E$ is isomorphic as a real curve to one of the Legendre form 
\begin{equation}\label{real_legendre_family}
y^2=x(x-1)(x-\la)
\end{equation}
where $\la$ is a real parameter, $\la\neq0,1$. We have checked that Conjecture~\ref{separability_conjecture} holds whenever $n \leq 9$ by explicitly computing the $x$-discriminant of $P_n$ for the elliptic curve \eqref{real_legendre_family} and checking that the only real roots are 0 and 1 (which appear with large multiplicities). Note that the discriminant of $P_9$ is a polynomial in $\la$ of degree 218.

\medskip
On the other hand, when $f$ has two conjugate roots, $P_n$ is no longer separable in general, whenever $n \geq 4$. For example, one checks that the discriminant of $P_4$ of the curve with normal form
\begin{equation}\label{non_legendre_family}
y^2=x(x-(1+\la i))(x-(1-\la i))
\end{equation}
has a nontrivial real root when $\la= \pm \frac{1}{\sqrt{3}}$. So Conjecture~\ref{separability_conjecture}, assuming it holds, strongly depends on the maximal reality of $E(\mb{R})$.

\medskip
Curiously, it also appears that the {\it degree} of the $x$-discriminant of $P_n$ depends upon the topological type of $E(\mb{R})$.

\begin{conj}\label{discriminant_degree_conjecture}
Let $f=f(x,\la)$ be a cubic polynomial associated to a one-parameter family of elliptic curves $y^2=f$ as in \eqref{real_legendre_family} or \eqref{non_legendre_family}. Define the associated key polynomials $P_n=P_n(x,\la)$ according to the recursive hierarchy \eqref{P_n_recursion} together with the seed \eqref{seed_datum}. When $f$ is of Legendre type \eqref{real_legendre_family} (resp., of type \eqref{non_legendre_family}), the degree of the $x$-discriminant of $P_n$ is equal to $3n^2-3n+2$ (resp., $4n^2-2n$).
\end{conj}


\medskip 
Hereafter, we focus our attention on elliptic curves $E=E(\la)$ arising from the real Legendre family \eqref{real_legendre_family}. We will be concerned with {\it generalized key polynomials} whose roots encode the $x$-coordinates of inflection points of $\mc{F}_g$ away from the ramification locus $R_{\pi}$ of $\pi: E \ra \mb{P}^1$, or equivalently, inflection points whose $x$-coordinates are not roots of $f$. By substituting $y^{(m)}= f^{-m}y P_m$ in the $n=(k+1)$-th Wronskian determinant that defines $\al|_{U_y}$, we can write $\al|_{U_y}=(f^{-n}y)^\mu P_{\mu,n}$, where $P_{\mu,n}$ is now the generalized key polynomial associated with $\mc{F}_g$ when $k-g=\mu\geq1$. This  $P_{\mu,n}$ is a polynomial of degree $\mu$ in the key polynomials $P_m=P_{1,m}$. We now make this explicit in the cases $\mu=2,3$.

\medskip 
{\bf Case: $\mu=2$.} Here
\begin{equation}\label{al_mu=2}
\al|_{U_y}= \det \left( \begin{array}{cc}
y^{(n)} & y^{(n+1)} \\
(xy)^{(n)} & (xy)^{(n+1)}
\end{array} \right)=(f^{-n}y)^2P_{2,n};
\end{equation}
applying the facts that
\begin{equation}\label{derivative_identities}
y^{(m)}= f^{-m}y P_m \text{ and } (xy)^{(m)}= my^{(m-1)}+ xy^{(m)}
\end{equation}
for all $m \geq 1$ in \eqref{al_mu=2} yields
\[
\al|_{U_y}=(n+1)(y^{(n)})^2- ny^{(n+1)}y^{(n-1)}
\]
and
\begin{equation}\label{P_2,n}
P_{2,n}= (n+1)P_n^2- nP_{n-1}P_{n+1}
\end{equation}
for all $n \geq 1$.

\medskip
{\bf Case: $\mu=3$.} This time, we have
\begin{equation}\label{al_mu=3}
\al|_{U_y}= \det \left( \begin{array}{ccc}
y^{(n)} & y^{(n+1)} & y^{(n+2)} \\
(xy)^{(n)} & (xy)^{(n+1)} & (xy)^{(n+2)} \\
(x^2y)^{(n)} & (x^2y)^{(n+1)} & (x^2y)^{(n+2)}
\end{array} \right)=(f^{-n}y)^3P_{3,n}.
\end{equation}
Evaluating the Wronskian determinant \eqref{al_mu=3} and applying the derivative identities \eqref{derivative_identities} together with the fact that
\begin{equation}\label{derivative_identity_bis}
(x^2y)^{(m)}= m(m-1)y^{(m-2)}+ 2mxy^{(m-1)}+ x^2y^{(m)}
\end{equation}
whenever $m \geq 2$ yields
{\small
\begin{equation}\label{P_3,n}
\begin{split}
P_{3,n}&= (n+1)^2 (n+2) P_n^3 - n (n+1)P_n (2 (n+2) P_{n-1} P_{n+1}
+ (n-1) P_{n-2} P_{n+2}) \\
&+  n ((n^2+n-2) P_{n-2} P_{n+1}^2 + n (n+1) P_{n-1}^2 P_{n+2})
\end{split}
\end{equation}
}
for all $n \geq 2$.

\section{Real roots of generalized key polynomials}
In his Ph.D. thesis \cite[Prop. 3.2.5]{Ga}, the second author obtained a complete characterization of the real inflectionary loci of complete real linear series $|\mc{L}_\mathbb{R}|$ of degree $d\geq2$ on a real elliptic curve $E$. He showed, in particular, that when the real locus $E(\mb{R})$ has two real connected components, the number of real inflection points of $|\mc{L}_\mathbb{R}|$ is either $d$ or $2d$, depending upon whether $d$ is odd or even. In the present setting, $d=2k$ is always even, and therefore $|\mc{L}_\mathbb{R}|$ has precisely $4k$ real inflection points.

\medskip
On the other hand, the linear series spanned by $\mc{F}_g$ is complete precisely when $n=\mu+2$, i.e. when $k=\mu+1$. In that situation, there are $4\mu+4$ real inflection points, of which $4\mu$ lie away from $R_{\pi}$. It follows that the corresponding generalized key polynomial $P_{\mu,\mu+2}=P_{\mu,\mu+2}(\la)$ associated with any choice of real parameter $\la$ has precisely $2\mu$ real roots $x=\gamma$ for which $f(\gamma,\la) > 0$.  

\medskip
Surprisingly, it seems that the generalized key polynomial $P_{\mu,\mu+2}$ determines the real inflectionary behavior of {\it every} generalized key polynomial $P_{\mu,n}=P_{\mu,n}(\la)$ whenever $n \geq \mu+2$, and irrespective of the choice of $\la$. Namely, graphing the zero loci $(P_{\mu,n}(x,\la)=0)$ when $\mu=1,2,3$ for low values of $n \geq \mu+2$, we are led to speculate the following.

\begin{conj}\label{numerology_of_key_polynomial_roots}
Let $\mu \geq 1$ and $n \geq \mu+2$ be nonnegative integers. 
For every fixed value of $\la\neq0,1$, the corresponding generalized key polynomial $P_{\mu,n}(\la)$ has precisely either $\mu$ or $2\mu$ real roots $x=\gamma$ such that $f(\gamma,\la) > 0$, depending upon whether $n-\mu$ is odd or even.
\end{conj}

Equivalently, the conjecture predicts that {\it the codimension-$(g-1)$ linear subseries of $|\mc{L}_\mathbb{R}|$ spanned by $\mc{F}_g$ has precisely either $2(k-g)$ or $4(k-g)$ real inflection points away from $R_{\pi}$, depending upon whether $g \geq 1$ is even or odd}.

\section{Real loci of plane curves defined by key polynomials}\label{real_loci}              
\subsection{Degrees and symmetries of key polynomials}
We begin by showing that $P_n=P_n(x,\la)$ is of degree exactly $2n$ in $x$ (a consequence of Pl\"ucker's formula for $n\geq4$) and of degree exactly $n$ in $\la$. 

\begin{lemma}\label{key_poly_degree}
For every positive integer $n \geq 1$, we have
\[
\deg_x(P_n)=2n\quad \text{ and } \quad \deg_\la(P_n)=n.
\]
\end{lemma}

\begin{proof}
The claim is clear for $n=1,2$; so we assume that $n \geq 3$. First recall that since $k=g+1$, we have $n=k+1=g+2$. So the geometric object of interest is a $g_{2g+2}^{g+3}$ on the elliptic curve $E$, whose total inflectionary degree may be realized as 
\begin{equation}\label{total_inflection}
    (2g+2)(g+3)=I_{R_{\pi}}+2\text{deg}_x(P_{g+2})
\end{equation}
where $I_{R_{\pi}}$ is the inflection concentrated in the ramification points, which is $4\binom{g+1}{2}+2(g-1)$. From \eqref{total_inflection} we deduce that $\text{deg}_x(P_{g+2})=2g+4=2(g+2)$, as desired.

To prove the second statement, we argue by induction. The claim is obvious when $n=1$; so assume that it holds for all $k\leq n$ for some $n>1$. We then have $$P_n=a_n(x)\lambda^n+Q_n(x,\lambda)$$ with $a_n(x)\in\mathbb{R}[x]$ and  deg$_\lambda(Q_n)<n$. Applying \eqref{P_n_recursion} to $P_n$ writen in this form, we see that the coefficient of $\lambda^{n+1}$ in $P_{n+1}$ is the polynomial $$a_{n+1}(x)=-x(x-1)\frac{da_{n}(x)}{dx}+(-n+1/2)(1-2x)a_n(x).$$
If $a_{n+1}(x)=0$, then $a_{n}(x)$ is a solution of the following differential equation $$\frac{y'}{y}=R(x)=\frac{(-n+1/2)(1-2x)}{x(x-1)}$$
but $y=Ke^{\int R(x)dx}=K(x^2-x)^{(2n-1)/2}$ is not a polynomial, so $a_{n+1}(x)\neq0$ and deg$_\lambda(P_{n+1})=n+1$.
\end{proof}

We next show that the curve defined by $P_n=0$ has special symmetry.
\begin{lemma}\label{key_poly_symmetry}
For every positive integer $n \geq 1$, we have
\begin{equation}\label{symmetry_property}
P_n(x,\la)= P_n(x,z) \text{ and }P_n(x+1,\la+1)= P_n(-x,-\la).
\end{equation}
Here by $P_n(x,z)$ we mean the polynomial of degree $2n$ obtained by homogenizing with respect to $z$ to obtain a degree-$2n$ polynomial in $\mb{R}[x,\la,z]$; and then dehomogenizing with respect to $\la$ to obtain a degree-$2n$ polynomial in $\mb{R}[x,z]$.
\end{lemma}

\begin{proof}
We prove \eqref{symmetry_property} by applying the basic recursion \eqref{P_n_recursion}. Clearly both properties \eqref{symmetry_property} hold when $n=1$. For the first property, note that $f(x,\la)=f(x,z)$; the fact that the analogous property holds for $P_n$ is now immediate by induction using \eqref{P_n_recursion}. We now focus on the on the second symmetry in \eqref{symmetry_property}. By induction, we may assume it holds for all $n \leq N$, for some $N \geq 1$. Then because (the second symmetry in) \eqref{symmetry_property} is preserved under taking products and multiplying by scalars, $(-n+\fr{1}{2}) P_n \cdot \fr{df}{dx}$ also satisfies \eqref{symmetry_property}. It is slightly more delicate, but nevertheless true, that $\fr{dP_n}{dx} \cdot f$ satisfies \eqref{symmetry_property}. Indeed, one checks easily that $f(x+1,\la+1)=-f(-x,-\la)$, so it suffices to check that the $x$-derivative of $P_n$ satisfies the same antisymmetric property, but this follows easily from the chain rule.
\end{proof}

An upshot of Lemma~\ref{key_poly_symmetry} is that the monodromy group associated to the projection from the point $[0:0:1]$ of the projective closure $\ov{\mc{C}(n)}$ of the curve $\mc{C}(n)$ defined by $P_n=0$ inside of $\mb{P}^2_{x,\la,z}$ contains transpositions that freely permute the points $p_1$, $p_2$, and $p_3$ with coordinates $[0:0:1]$, $[0:1:0]$, and $[1:1:1]$. In particular, the singularities of $\ov{\mc{C}(n)}$ in these three points are analytically isomorphic.

\begin{conj}\label{singularity_conj}
For every positive integer $n \geq 1$, the plane curve $\ov{\mc{C}(n)}$ is nonsingular along $\mb{P}^2 \setminus \{p_1,p_2,p_3\}$.
\end{conj}

Note that Conjecture~\ref{singularity_conj} is easy to verify by computer (i.e. algorithmically, using Groebner bases) whenever $n$ is small. However, we have thus far been unable to show ``by hand" that the system of equations $P_n(x,\lambda)= \frac{dP_n}{dx}= \frac{dP_n}{d\lambda}$ is exactly solved by the coordinate pairs $(0,0)$ and $(1,1)$ in general. The next four subsections are devoted to proving Conjecture~\ref{numerology_of_key_polynomial_roots} in the first nontrivial cases.

\begin{thm} Conjecture~\ref{numerology_of_key_polynomial_roots} holds when $\mu=1$ and $n=2,3,4,5$.
\end{thm}

In fact, the case $n=2$ is non-geometric (since $k=1$ in that case) and when $n=3$ the associated linear series is complete (so in particular Conjecture~\ref{singularity_conj} holds on the basis of \cite[Thm 3.2.5]{Ga}), so the first geometrically interesting case  $n=4$. However as the argument used in the proof when $n=4$ is a more elaborate version of the argument for $n=3$ (which is itself a more elaborate version of that used for $n=2$) we find it instructive to include a discussion of the $n=2$ and $n=3$ cases as well. The key point here is that the recursion~\eqref{P_n_recursion} implies that the $n$th key polynomial $P_n=P_n(x,\la)$, which is of degree  $2n$ in $x$, is only of degree $n$ in $\la$. 
Accordingly we can obtain a {\it global} parametrization $\la=\la(x)$ for $\la$ in terms of $x$ whenever $n \leq 4$.

\subsection{Case: $n=2$} We begin by writing $P_2$ as a quadratic polynomial in $\la$ with coefficients in $\mb{R}[x]$, namely
\[
P_2(x,\la)= -\fr{1}{4} \la^2+ \bigg(-x^3+ \fr{3}{2}x^2 \bigg)\la+ \bigg(-x^3+ \fr{3}{4}x^4 \bigg). 
\]
Applying the quadratic formula, we deduce that
\[
\la=\la(x)= (-2x^3+3x^2) \mp 2(x(x-1))^{3/2}
\]
solves $P_2=0$. Conjecture~\ref{numerology_of_key_polynomial_roots} in this particular case follows easily.

\subsection{Case: $n=3$} We have
\[
P_3(x,\la)= \bigg(-\fr{3}{4}x+\fr{3}{8}\bigg) \la^3+ \bigg(\fr{15}{8}x^2- \fr{3}{4}x \bigg)\la^2+ \bigg(\fr{3}{4}x^5-\fr{15}{8}x^4 \bigg)\la+ \bigg(-\fr{3}{8}x^6+ \fr{3}{4}x^5 \bigg).
\]
Whenever $x \neq \fr{1}{2}$, the leading coefficient $-\fr{3}{4}x+\fr{3}{8}$ of $P_3$ is nonzero. Assuming this to be the case, we may divide by the leading coefficient. Doing so produces a monic cubic $\wt{P}_3$ in $\la$ whose coefficients are (formally) power series in $x$, namely
\[
\wt{P}_3(x,\la)= \la^3+ \bigg(-\fr{5}{2}+ \fr{1}{2}\al\bigg)x\la^2- (1+4\al)x^4 \la+ \bigg(\fr{1}{2}+ \fr{3}{2}\al\bigg)x^5
\]
where $\al= \fr{1}{1-2x}$. Now set $a_i:= [\la^i]\wt{P}_3$, $i=0,1,2$. Changing variables according to 
$\la \mapsto \la^{\ast}$ where 
\[
\la^{\ast}= \la+ \fr{a_2}{3}= \la+ \bigg(-\fr{5}{6}+\fr{1}{6}\al\bigg)x
\]
allows $\wt{P}_3$ to be rewritten as a {\it depressed cubic}, namely
\[
\wt{P}^{\ast}_3= \wt{P}^{\ast}_3(x,\la^{\ast})= (\la^{\ast})^3+ p\la^{\ast}+ q 
\]
where
{\small
\[
\begin{split}
p&=a_1- \fr{a_2^2}{3}= -(1+4\al)x^4+ \bigg(-\fr{25}{12}+ \fr{5}{6}\al- \fr{1}{12}\al^2 \bigg)x^2
\text{ and } \\
q&= a_0- \fr{a_1a_2}{3}+ \fr{2a_2^3}{27}= \fr{1}{3}(2\al^2-5\al-1)x^5+ \fr{1}{108}(\al-5)^3 x^3. 
\end{split}
\]
}
Cardano's formula now establishes that
\[
\la^{\ast}= u+v
\]
solves $\wt{P}^{\ast}_3=0$, where
\begin{equation}\label{u_and_v}
u= \sqrt[3]{-\fr{q}{2}+ \sqrt{\bigg(\fr{q}{2}\bigg)^2+ \bigg(\fr{p}{3}\bigg)^3} }
\text{ and }v= \sqrt[3]{-\fr{q}{2}- \sqrt{\bigg(\fr{q}{2}\bigg)^2+ \bigg(\fr{p}{3}\bigg)^3}}.
\end{equation}
Note that when 
\[
\De=\De(x):=\bigg(\fr{q}{2}\bigg)^2+ \bigg(\fr{p}{3}\bigg)^3
\]
is negative, the expression for $u$ in \eqref{u_and_v} selects for an arbitrary cube root, and $v:=\ov{u}$. In the latter situation, the corresponding cubic $\wt{P}^{\ast}_3= \wt{P}^{\ast}_3(x)$ has three real roots.

In our case, we have $\De= \fr{1}{432}x^8\al^4 R$, where
{\small
\[
\begin{split}
R&= -256(x-1)^8.
 \end{split}
 \]
 }
In particular, {\it $\De$ is negative whenever $x \notin \{0,1, \fr{1}{2} \}$}. From \eqref{u_and_v} we deduce that
\begin{equation}\label{u_simplified}
u= \fr{x(x-1)}{x-\fr{1}{2}} \cdot  \sqrt[3]{\fr{1}{18}\bigg(x-\fr{1}{2} \bigg)^2+\fr{5}{216}+ \fr{1}{3 \sqrt{3}} (x(x-1)) \bigg(x - \fr{1}{2} \bigg) \sqrt{-1}}.
\end{equation}

Further, by negativity of the discriminant $\De$, we may always distinguish three distinct roots of $\wt{P}^{\ast}_3(x)$ given by $u+ \ov{u}$, $\zeta u + \ov{\zeta u}$, and $\zeta^2 u+ \ov{\zeta^2 u}$, where $\zeta$ is a primitive cube root of unity and $u$ is as in \eqref{u_simplified}, assuming that an arbitrary cubic root of $-q/2+ \sqrt{\De}$ is fixed.

\medskip
Now let
\[
\wt{v}=\wt{v}(x):= \fr{1}{18}\bigg(x-\fr{1}{2} \bigg)^2+\fr{5}{216}+ \fr{1}{3 \sqrt{3}} (x(x-1)) \bigg(x - \fr{1}{2} \bigg) \sqrt{-1}.
\]
We can rewrite $\wt{v}$ as $\cos(\ga)+ \sin(\ga) \sqrt{-1}$, where
\[
\ga=\ga(x^{\ast}):= \arctan \Bigg(\fr{6}{\sqrt{3}} \cdot \fr{x^{\ast}((x^{\ast})^2-\fr{1}{4})}{(x^{\ast})^2+\fr{5}{12}} \Bigg)
\]
and where $x^{\ast}=x- \fr{1}{2}$. Putting everything together, we obtain three distinguished solutions $\wt{\la}_i=\wt{\la}_i(x^{\ast})$, $i=1,2,3$ of $P_3=P_3(x^{\ast})=0$, given by
{\small
\[
\begin{split}
\wt{\la}_1&= \bigg(\fr{(x^{\ast})^2-\fr{1}{4}}{x^{\ast}}\bigg) \cdot 2 \cos \bigg(\arctan \bigg(\fr{6}{\sqrt{3}} \cdot \fr{x^{\ast}((x^{\ast})^2-\fr{1}{4})}{(x^{\ast})^2+\fr{5}{12}} \bigg)\bigg/3\bigg)+ \fr{5}{6}x^{\ast}+ \fr{1}{2}+ \fr{1}{24} \cdot \fr{1}{x^{\ast}}, \\
\wt{\la}_2&= \bigg(\fr{(x^{\ast})^2-\fr{1}{4}}{x^{\ast}}\bigg) \cdot 2 \cos \bigg(\arctan \bigg(\fr{6}{\sqrt{3}} \cdot \fr{x^{\ast}((x^{\ast})^2-\fr{1}{4})}{(x^{\ast})^2+\fr{5}{12}} \bigg)\bigg/3+ 2\pi/3\bigg)+ \fr{5}{6}x^{\ast}+ \fr{1}{2}+ \fr{1}{24} \cdot \fr{1}{x^{\ast}}, \text{ and } \\
\wt{\la}_3&= \bigg(\fr{(x^{\ast})^2-\fr{1}{4}}{x^{\ast}}\bigg) \cdot 2 \cos \bigg(\arctan \bigg(\fr{6}{\sqrt{3}} \cdot \fr{x^{\ast}((x^{\ast})^2-\fr{1}{4})}{(x^{\ast})^2+\fr{5}{12}} \bigg)\bigg/3+ 4\pi/3\bigg)+ \fr{5}{6}x^{\ast}+ \fr{1}{2}+ \fr{1}{24} \cdot \fr{1}{x^{\ast}}.
\end{split}
\]
}

Conjecture~\ref{numerology_of_key_polynomial_roots} now follows in the present case from the following subsidiary claims about the real locus of the affine plane curve $\mc{C}(3)$.
\begin{enumerate}
\item When $x<0$, the real locus is comprised of three infinite $x$-monotone (i.e. nondecreasing with respect to $x$) curves $\la_i=\la_i(x)$, $i=1,2,3$ with $\la_1>0>\la_2>x>\la_3$. These curves intersect in the limit point $(0,0)$ (whose associated value of $\la$ describes a singular rational curve). Moreover, $\la_3$ has a vertical asymptote at $x=\fr{1}{2}$; and 
\[
-\infty=\lim_{x \ra -\infty} \la_2= \lim_{x \ra -\infty} \la_3= -\lim_{x \ra -\infty} \la_1.
\]
\item When $0<x<1$, the real locus is a union of monotone curves $\la_i$, $4 \leq i \leq 7$. Here $\la_4>x>\la_5>\la_6$; $\la_4$, $\la_5$, and $\la_6$ intersect in $(0,0)$; $\la_4$ and $\la_5$ connect $(0,0)$ with $(1,1)$, where they also intersect $\la_7$. Moreover, we have $\la_7>\la_4$ and
\[
\lim_{x \ra \fr{1}{2}^-} \la_6= -\lim_{x \ra \fr{1}{2}^+} \la_7= -\infty.
\]
\item When $x>1$, the real locus is comprised of three infinite monotone curves $\la_i$, $i=8,9,10$, where $\la_8>x>\la_9>\la_{10}$. These intersect in their common limit point $(1,1)$, in which the associated curve $E(\la)$ is singular. Moreover, we have
\[
\infty= \lim_{x \ra \infty} \la_8= \lim_{x \ra \infty} \la_9=-\lim_{x \ra \infty} \la_{10}
\]
\end{enumerate}

\qed

\subsection{Case: $n=4$} Here our analysis closely follows that of the $n=3$ case. To begin, we have
{\tiny
\[
P_4(x,\la)= \bigg(-3x^2+3x-\fr{15}{16} \bigg) \la^4+ \bigg(\fr{21}{2}x^3- \fr{39}{4}x^2+ 3x \bigg) \la^3+ \bigg(-\fr{105}{8}x^4+ \fr{21}{2}x^3- 3x^2 \bigg)\la^2+ \bigg(-\fr{3}{2}x^7+ \fr{21}{4} x^6 \bigg) \la+ \bigg(\fr{9}{16} x^8- \fr{3}{2}x^7 \bigg).
\]
}
Now set $\be=\fr{1}{x^2-x+\fr{5}{16}}= \fr{1}{(x-\fr{1}{2})^2+ \fr{1}{16}}$. Note that $\be$ is nonzero for all real values of $x$. Dividing $P_4$ by its leading coefficient yields
{\tiny
\[
\begin{split}
\wt{P}_4(x,\la)&= \la^4+ \bigg(-\fr{7}{2}x- \fr{1}{4}+ \bigg(-\fr{5}{32}x + \fr{5}{64} \bigg)\be\bigg)\la^3+ \bigg(\fr{35}{8}x^2+ \fr{7}{8}x+ \fr{65}{128}+ \bigg(\fr{15}{64}x- \fr{325}{2048}\bigg)\be\bigg)\la^2 \\
&+ \bigg(\fr{1}{2}x^5- \fr{5}{4}x^4- \fr{45}{32}x^3- \fr{65}{64}x^2- \fr{295}{512}x- \fr{265}{1024}+ \bigg(-\fr{645}{8192}x+ \fr{1325}{16384}\bigg)\be\bigg)\la \\
&+ \bigg(-\fr{3}{16}x^6+ \fr{5}{16}x^5+ \fr{95}{256}x^4+ \fr{35}{128}x^3+ \fr{645}{4096}x^2+ \fr{295}{4096}x+ \fr{1495}{65536}+ \bigg(\fr{5}{16384}x- \fr{7475}{1048576}\bigg)\be\bigg).
\end{split}
\]
}
Let $b_i:= [\la^i]\wt{P}_4(x,\la)$, $0 \leq i \leq 3$. Changing variables according to $\la \mapsto \la^{\ast}$ with $\la= \la^{\ast}-\fr{b_3}{4}$ enables us to convert $\wt{P}_4(x,\la)$ to a depressed quartic $\wt{P}^{\ast}_4(x,\la^{\ast})$ of the form $(\la^{\ast})^4+ p(\la^{\ast})^2+q\la^{\ast}+r$, where
\[
p= b_2- \fr{3}{8}b_3^2, q= \fr{1}{8}b_3^3- \fr{1}{2}b_2b_3+ b_1, \text{ and } r= -\fr{3}{256}b_3^4+ b_0- \fr{1}{4}b_1b_3+ \fr{1}{16}b_2b_3^2.
\]
According to Ferrari, the roots of $\wt{P}^{\ast}_4(x,\la^{\ast})$ are given by
\begin{equation}\label{P4solution}
\la^*=\fr{ \pm_1 \sqrt{2m} \pm_2 \sqrt{-(2p+2m \pm_1 \fr{\sqrt{2}q}{\sqrt{m}})}}{2}
\end{equation}
where $m$ denotes any root of the resolvent cubic
\[
K(z)= 8z^3+ 8pz^2+ (2p^2-8r)z-q^2
\]
associated with $\wt{P}^{\ast}_4$. Here the monic depressed cubic associated with $K$ is 
{\small
\[
\begin{split}
K_0 &= z^3-\bigg(\fr{1}{12}p^2+r \bigg)z+ \bigg(-\fr{1}{108}p^3+ \fr{1}{3}rp-\fr{1}{8}q^2 \bigg) \\
&= z^3+ \bigg(-b_0- \fr{1}{12}b_2^2+ \fr{1}{4}b_1b_3 \bigg)z+ \bigg(-\fr{1}{8}b_1^2+ \fr{1}{3}b_0b_2- \fr{1}{108}b_2^3+ \fr{1}{24} b_1b_2b_3- \fr{1}{8} b_0b_3^2 \bigg).
\end{split}
\]
}
Now let $p_0:=[z]K_0$ and $q_0:=[z^0]K_0.$ From Cardano it follows that roots of $K_0$ are given by $m_0=u+v$, where
\[
u=\sqrt[3]{-\fr{q_0}{2}+ \sqrt{\De_0}}
\text{ and }v= \sqrt[3]{-\fr{q_0}{2}- \sqrt{\De_0}}.
\]
Here $\De_0:=(\fr{q_0}{2})^2+ (\fr{p_0}{3})^3$ is the discriminant of $K_0$. In this case
{\small
\[
q_0= -\fr{\be^3}{2^5} x^6 (x-1)^6 \bigg(\bigg(x-\fr{1}{2}\bigg)^4+ \fr{5}{6}\bigg(x-\fr{1}{2}\bigg)^2+ \fr{11}{432}\bigg) \text{ and } \De_0= \fr{\be^6}{2^{12} } x^{15}(x-1)^{15} \bigg(\bigg(x-\fr{1}{2}\bigg)^2+\fr{5}{108}\bigg).
\]
}

\medskip
It follows that
$u=\fr{\be}{4} x^2(x-1)^2 \sqrt[3]{U_1+ \sqrt{U_2}}$ and $v=\fr{\be}{4} x^2(x-1)^2 \sqrt[3]{U_1- \sqrt{U_2}}$
where
\[
U_1:=\bigg(x-\fr{1}{2}\bigg)^4+ \fr{5}{6}\bigg(x-\fr{1}{2}\bigg)^2+ \fr{11}{432} \text{ and } U_2:=x^3(x-1)^3\bigg(\bigg(x-\fr{1}{2}\bigg)^2+\fr{5}{108}\bigg).
\]
Here $U_1$ is positive for all real values of $x$, while $U_2$ is positive (resp., negative) whenever $x<0$ or $x>1$ (resp., whenever $0<x<1$). Moreover, we have
\begin{equation}\label{U1squaredminusU2}
U_1^2- U_2= \fr{64}{27} \bigg(\bigg(x-\fr{1}{2}\bigg)^2+ \fr{1}{12}\bigg)^3
\end{equation}
which is clearly positive for all real-valued $x$. Accordingly, whenever $x<0$ or $x>1$, we take $\sqrt{U_2(x)}$ to mean the positive square root of $U_2(x)$; it follows that $U_1-\sqrt{U_2}$ and $U_1+\sqrt{U_2}$ are positive. In this situation, we denote by $u=u(x)$ and $v=v(x)$ the real functions that select for $\fr{\be}{12} x^2(x-1)^2$ times the unique real positive cubic roots of $U_1+ \sqrt{U_2}$ and $U_1- \sqrt{U_2}$, respectively. Similarly, whenever $0<x<1$, we take $\sqrt{U_2(x)}$ to mean $\sqrt{-1}$ times the positive square root of $-U_2(x)$, and then $U_1+ \sqrt{U_2}$ and $U_1- \sqrt{U_2}$ become complex conjugates of one another, with positive real parts. Finally, we obtain
{\small
\begin{equation}\label{m_expression}
\begin{split}
m&= m_0-\fr{1}{3}p \\
&= \fr{\be}{4} x^2(x-1)^2 \sqrt[3]{U_1+ \sqrt{U_2}}+ \fr{\be}{12} x^2(x-1)^2 \sqrt[3]{U_1- \sqrt{U_2}}- \fr{1}{3}b_2+ \fr{1}{8}b_3^2 \\
&= \fr{\be}{4} x^2(x-1)^2 \bigg(\ga+ \fr{7\be}{24} \bigg(\bigg(x-\fr{1}{2}\bigg)^2+ \fr{1}{28}\bigg)\bigg).
\end{split}
\end{equation}
}
where $\ga= \ga(x):=\sqrt[3]{U_1+ \sqrt{U_2}}+ \sqrt[3]{U_1- \sqrt{U_2}}$. Note here that $\ga$ satisfies the algebraic equation $\ga^3-2U_1- 4((x-\fr{1}{2})^2+ \fr{1}{12})\ga=0.$

\medskip
It follows from the above discussion that our choice of $m$ in \eqref{m_expression} is always positive, and it remains to control for the sign of the expression inside the second radical in \eqref{P4solution}. 
Here
{\small
\[
p+m= m_0+ \fr{2}{3}p=\fr{\be}{4} x^2(x-1)^2 \bigg(\ga- \fr{7\be}{12}\bigg(\bigg(x-\fr{1}{2}\bigg)^2+ \fr{1}{28}\bigg) \bigg)
\]
}
while
\[
q=\fr{\be^3}{2} \bigg(x-\fr{1}{2}\bigg) x^3 (x-1)^3 \bigg(\bigg(x-\fr{1}{2}\bigg)^4+ \fr{7}{32}\bigg(x-\fr{1}{2}\bigg)^2+ \fr{1}{128}\bigg).
\]

We claim that
\begin{enumerate}
\item[(i)] $-(\sqrt{2m}(p+m)-q) \geq 0 \iff x \geq 1$, and
\item[(ii)] $-(\sqrt{2m}(p+m)+q) \geq 0 \iff x \leq 0$.
\end{enumerate}
Indeed, on the basis of the properties of the key polynomials $P_n$ and their derivatives for small values of $n$ (namely, the validity of Conjecture~\ref{singularity_conj}) it suffices to check that the signs of $-(\sqrt{2m}(p+m) \pm q)$ are as claimed {\it locally} near $x=0$ and $x=1$. This may be achieved by computing the values of $\sqrt{2m}(p+m) \pm q$ and its first derivative in $x=0$ and $x=1$. We omit the explicit calculation.

\medskip
Finally, the explicit ``parametrization" \eqref{P4solution} for $\mc{C}(4)$ may be used to check that the real locus $\mc{C}(4)(\mb{R})$ has the following features, and the validity of Conjecture~\ref{numerology_of_key_polynomial_roots} when $n=4$ follows easily.

\begin{enumerate}
\item There are no points in the real locus for which $0<x<1$.
\item When $x<0$, there are two monotone curves $\la_i=\la_i(x)$, $i=1,2$ for which $\la_1>0>\la_2>x$. These curves intersect in their common limit of $(0,0)$, and satisfy
\[
\lim_{x \ra \infty} \la_1= \infty= -\lim_{x \ra \infty} \la_2.
\]
\item When $x>1$, there are two monotone curves $\la_i=\la_i(x)$, $i=3,4$ for which $x>\la_3(x)>\la_4(x)$. These curves intersect in their common limit of $(1,1)$, and satisfy
\[
\lim_{x \ra \infty} \la_3= \infty= -\lim_{x \ra \infty} \la_4.
\]

\end{enumerate}


\subsection{Case: $n=5$} This case is more difficult than the preceding ones. Indeed, as noted in Section~\ref{real_loci}, the monodromy group $G$ of $P_5(x,\la)$ over $\mb{C}(x)$ associated to the projection of $\ov{\mc{C}(5)}$ from $(0,0,1)$ contains an element of order 3. On the other hand, the solvable subgroups of $S_5$ are precisely the subgroups of the Frobenius group $F_{20}$ of order 20. It follows that $G$ is not solvable, and therefore no global solution $\la=\la(x)$ of the plane curve with equation $P_5(x,\la)=0$ is available. On the other hand, {\it local} (Puiseux) parametrizations $\la=\la(x)$ that solve $P_n(x,\la)=0$ are always available (see, e.g. \cite[Sec. 8.3]{BK}), and we will carry this out now when $n=5$ near the singular point $p_1$; in view of \eqref{key_poly_symmetry}, the local descriptions of $\ov{\mc{C}(5)}$ near the singular points $p_1$, $p_2$, and $p_3$ are identical.

\medskip
{\fl \bf Local parametrizations at $(0,0)$.} We have
{\tiny
\begin{equation}\label{P5}
\begin{split}
P_5(x,\la)&= - \fr{45}{32} x^{10} +\fr{75}{16} x^9 - \fr{675}{32} x^8 \la + \fr{75}{16} x^9 \la - 
 15 x^3 \la^2 + \fr{135}{2} x^4 \la^2 - \fr{945}{8} x^5 \la^2 + 
 \fr{1575}{16} x^6 \la^2 + \fr{45}{2} x^2 \la^3 - \fr{195}{2} x^3 \la^3 \\
 &+
     \fr{2565}{16} x^4 \la^3 - \fr{945}{8} x^5 \la^3 - \fr{225}{16} x \la^4 +
 \fr{1935}{32} x^2 \la^4 - \fr{195}{2} x^3 \la^4 + \fr{135}{2} x^4 \la^4 + 
 \fr{105}{32} \la^5 - \fr{225}{16} x \la^5 + \fr{45}{2} x^2 \la^5 - 15 x^3 \la^5.
\end{split}
\end{equation}
}
From \eqref{P5} it follows that the lower hull of the Newton polygon consists of the line segments $\ell_1, \ell_2$ between $(0,5)$ and $(3,2)$ (of slope $-1$); and $(3,2)$ and $(9,0)$ (of slope $-\fr{1}{3}$)
, respectively. So there are at least two local (complex) branches near $(0,0)$, whose respective multiplicities are equal to the absolute values of the slopes of the corresponding segments $\ell_i$.

{\bf \fl Branches associated with $\ell_1$.} Begin by writing
\begin{equation}\label{first_approximation}
\la_1= x(c_1+ \la_1^1)
\end{equation}
where $c_1 \in \mb{C}^{\ast}$ and 
\[
\la_1^1=\la_1^1(x)= c_2x^{\ga_2}+ c_3x^{\ga_2+\ga_3}+ \dots
\]
is a Puiseux series that remains to be determined. Now substitute the first approximation \eqref{first_approximation} into the implicit equation $P_5=0$ and divide by the smallest common power in $x$; the result is a polynomial in $x$, $c_1$, and $\la_1^1$ the sum of whose lowest $x$-order terms is
\[
\fr{15}{32}x^5 c_1^2 (7c_1^3-30c_1^2+48c_1-32)= \fr{15}{32}x^5 c_1^2 (c_1-2)(7c_1^2-16c_1+16).
\]
As the lowest $x$-order terms of $P_5$ must sum to zero, we deduce that either $c_1=2$ or else $c_1$ is one of the two complex conjugate roots of $7c_1^2-16c_1+16$. In the latter case, we obtain two non-real branches, exchanged under conjugation. So say that $c_1=2$. Then upon dividing by $\fr{15}{32}x^5$ we obtain
{\tiny
\begin{equation}\label{second_approximation}
\begin{split}
P_5^1(x,\la_1^1)&= 16 x - 64 x^2 + 104 x^3 - 80 x^4 + 17 x^5 + 48 \la_1 - 192 x \la_1 + 280 x^2 \la_1 - 136 x^3 \la_1 - 45 x^4 \la_1 + 10 x^5 \la_1 + 96 \la_1^2 - 
 408 x \la_1^2 + 648 x^2 \la_1^2 \\
 &- 406 x^3 \la_1^2 + 88 \la_1^3 - 376 x \la_1^3 + 
 598 x^2 \la_1^3 - 380 x^3 \la_1^3 + 40 \la_1^4 - 171 x \la_1^4 + 272 x^2 \la_1^4 - 
 176 x^3 \la_1^4 + 7 \la_1^5 - 30 x \la_1^5 + 48 x^2 \la_1^5 - 32 x^3 \la_1^5.
\end{split}
\end{equation}
}
The lower hull of the Newton polygon of $P_5^1(x,\la_1^1)$ is a single line segment of slope $-1$ linking $(0,1)$ and $(1,0)$. It follows immediately that $\ga_i=1$ for all $i \geq 2$. Substituting $\la_1^1= \sum_{i=1}^{\infty} c_{i+1}x^i$ in \eqref{second_approximation} and setting the result equal to zero, we may compute all of the remaining power series coefficients $c_i$, $i \geq 2$. Clearly the branch in question is real.

{\bf \fl Branches associated with $\ell_2$.} Consider now those branches arising from the line segment $\ell_2$. Much as before, begin by writing
\begin{equation}\label{first_approximation_bis}
\la_1= x^3(c_1+ \la_1^1)
\end{equation}
where $\la_1^1=\la_1^1(x)= c_2x^{\ga_2}+ c_3x^{\ga_2+\ga_3}+ \dots$ is a Puiseux series that remains to be determined. Substituting \eqref{first_approximation_bis} into the implicit equation $P_5=0$ and dividing by the smallest common power in $x$ yields a polynomial in $x$, $c_1$, and $\la_1^1$ whose sum of lowest $x$-order terms is
\[
\fr{15}{32}x^9(10-32c_1^2).
\]
It follows that $c_1=\pm \fr{\sqrt{5}}{4}$. Now say that $c_1=\fr{\sqrt{5}}{4}$. Then dividing by $-\fr{15}{32768}x^9$ we obtain a polynomial 
{\tiny
\[
\begin{split}
P_5^2(x,\la_1^1)&=-43008 x + 80640 x^2 + 7680 \sqrt{5} x^2 - 67200 x^3 + 
 14080 \sqrt{5} x^3 + 3000 x^4 - 27360 \sqrt{5} x^4 - 12900 x^5 
 \\
 &+ 20160 \sqrt{5} x^5 + 20800 x^6 - 175 \sqrt{5} x^6 - 14400 x^7 + 
 750 \sqrt{5} x^7 - 1200 \sqrt{5} x^8 + 800 \sqrt{5} x^9 \\
 &+ 16384 \sqrt{5} \la_1 - 73728 \sqrt{5} x \la_1 + 129024 \sqrt{5} x^2 \la_1
 + 189440 x^3 \la_1 - 107520 \sqrt{5} x^3 \la_1 - 328320 x^4 \la_1 \\
 &+ 9600 \sqrt{5} x^4 \la_1 + 241920 x^5 \la_1 - 41280 \sqrt{5} x^5 \la_1 - 
 3500 x^6 \la_1 + 66560 \sqrt{5} x^6 \la_1 + 15000 x^7 \la_1 \\
 &- 46080 \sqrt{5} x^7 \la_1 - 24000 x^8 \la_1 + 16000 x^9 \la_1 + 32768 \la_1^2 - 
 147456 x \la_1^2 + 258048 x^2 \la_1^2 - 36864 \sqrt{5} x^2 \la_1^2 \\
 &-215040 x^3 \la_1^2 + 159744 \sqrt{5} x^3 \la_1^2 + 57600 x^4 \la_1^2 - 
 262656 \sqrt{5} x^4 \la_1^2 - 247680 x^5 \la_1^2 \\
 &+193536 \sqrt{5} x^5 \la_1^2 + 399360 x^6 \la_1^2 - 5600 \sqrt{5} x^6 \la_1^2 - 
 276480 x^7 \la_1^2 + 24000 \sqrt{5} x^7 \la_1^2 - 38400 \sqrt{5} x^8 \la_1^2 \\
 &+ 25600 \sqrt{5} x^9 \la_1^2 - 49152 x^2 \la_1^3 + 212992 x^3 \la_1^3 - 
 350208 x^4 \la_1^3 + 30720 \sqrt{5} x^4 \la_1^3 + 258048 x^5 \la_1^3 \\
 &-132096 \sqrt{5} x^5 \la_1^3 - 22400 x^6 \la_1^3 + 212992 \sqrt{5} x^6 \la_1^3 + 
 96000 x^7 \la_1^3 - 147456 \sqrt{5} x^7 \la_1^3 - 153600 x^8 \la_1^3 \\
 &+ 102400 x^9 \la_1^3 + 30720 x^4 \la_1^4 - 132096 x^5 \la_1^4 + 
 212992 x^6 \la_1^4 - 8960 \sqrt{5} x^6 \la_1^4 - 147456 x^7 \la_1^4 \\
 &+ 38400 \sqrt{5} x^7 \la_1^4 - 61440 \sqrt{5} x^8 \la_1^4
 + 40960 \sqrt{5} x^9 \la_1^4 - 7168 x^6 \la_1^5 + 30720 x^7 \la_1^5 - 
 49152 x^8 \la_1^5 + 32768 x^9 \la_1^5
\end{split}
\]
}
whose Newton polygon's lower hull is a single line segment of slope $-1$ linking $(0,1)$ and $(1,0)$. It follows just as before that $\ga_i=1$ for all $i \geq 2$, and that all remaining power series coefficients $c_i$, $i \geq 2$ may be computed by substituting $\la_1^1= \sum_{i=1}^{\infty} c_{i+1}x^i$ in $P_5^2(x,\la_1^1)=0$. The corresponding branch is real.

\medskip
Finally, when $c_1=-\fr{\sqrt{5}}{4}$ we obtain a polynomial $P_5^3(x,\la_1^1)= P_5^2(x,-\la_1^1)$. The corresponding branch is clearly real as well. 

\medskip
To conclude, it now suffices to show that $\mc{C}(5)(\mb{R})$ has the same characteristic features as in the $n=3$ case; these collectively determine a {\it topological profile} for the real locus, and we conjecture that the following is true.

\begin{conj}\label{topological_profile_of_real_locus}
Let $n \geq 2$ be a positive integer. Then $\mc{C}(n)(\mb{R})$ and $\mc{C}(n \mod 2 + 2)(\mb{R})$ have the same topological profile.
\end{conj}

Clearly Conjecture~\ref{topological_profile_of_real_locus} implies Conjecture~\ref{numerology_of_key_polynomial_roots}. For the sake of completeness, we also record the following conjecture, which is related to Conjectures~\ref{topological_profile_of_real_locus} and \ref{singularity_conj}.

\begin{conj}\label{lambda_discriminant_roots}
For all $n \geq 2$, the set of real roots of the discriminant $\De_n=\De_n(x)$ of $P_n=P_n(x,\la)$ with respect to $\la$ is $\{0,1\}$.
\end{conj}

Figure 1 includes computer-generated sketches of the real loci of $\mc{C}(n)$ for some even values of $n$; the graphs give compelling evidence for Conjectures~\ref{topological_profile_of_real_locus} and \ref{numerology_of_key_polynomial_roots}.

\begin{figure}[!htb]
 \includegraphics{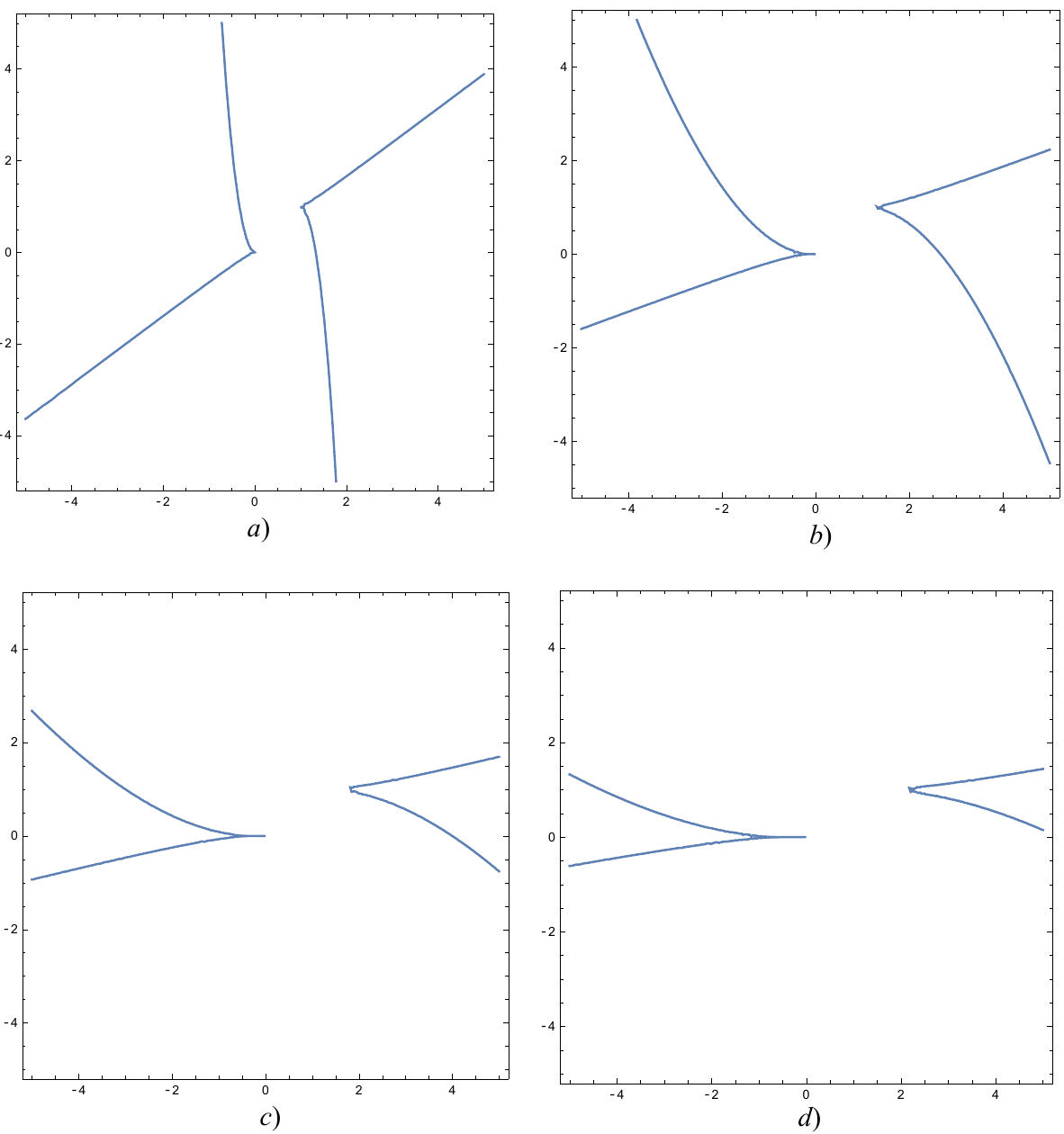}
 \caption{Real loci of a) $P_2=0$, b) $P_4=0$, c) $P_6=0$, and d) $P_8=0$.}
 \end{figure}

The affirmation of Conjecture~\ref{lambda_discriminant_roots} is a consequence by the monotonicity property of Conjecture~\ref{topological_profile_of_real_locus}, together with the conjectural characterization \ref{singularity_conj} of the real singular points of $\ov{\mc{C}(n)}$. In particular, Conjecture~\ref{lambda_discriminant_roots} holds when $n=5$.

\medskip
It remains to verify Conjecture~\ref{topological_profile_of_real_locus} when $n=5$. To this end, first note that implicitly differentiating the equation $P_n(x,\la(x))=0$ with respect to $x$ shows that
\[
\frac{d\la}{dx}=0 \iff \frac{dP_n}{dx}=0
\]
away from the singular points of $\mc{C}(n)$. On the other hand, for small values of $n$ (including $n=5$), the equations $P_n= \frac{dP_n}{dx}=0$ have no common affine solutions $(x,y)$ in $\mb{R}^2 \setminus \{(0,0),(1,1)\}$. In light of the Puiseux parametrizations of the real branches near $(0,0)$ and the symmetries \eqref{key_poly_symmetry}, the desired {\it $x$-monotonicity} properties of the solution curves $\la=\la(x)$ follow immediately. It suffices now to check that $P_5(x,\la)=0$ has exactly three solutions $\la$ for each fixed value of $x$ in $\mb{R} \setminus \{0,1,\fr{1}{2}\}$, i.e. that $\mc{C}(5)(\mb{R})$ is connected over each of the critical intervals $(-\infty,0)$, $(0,1)$, and $(1,\infty)$ in $x$. But this, in turn, is an immediate consequence of the Implicit Function theorem. 

\medskip
Finally, it remains to verify the inequalities of the real solution curves $\la=\la(x)$ described in the topological profile of $\mc{C}(3)$ above relative to $x$. These hold locally near the singular points $(0,0)$ and $(1,1)$ of $\mc{C}(5)$ in the affine plane, so it suffices simply to show that the intersection of the line $y=x$ with $\mc{C}(5)(\mb{R})$ is precisely supported along $\{(0,0),(1,1)\}$. And indeed, we have
\[
P_5(x,x)=-\fr{105}{32} x^5 + \fr{525}{32} x^6 - \fr{525}{16} x^7 + \fr{525}{16} x^8 - \fr{525}{32} x^9 + \fr{105}{32} x^{10}= -\fr{105}{32}x^5(x-1)^5.
\]
\qed
\par
\medskip
Figure 2 includes sketches of the the real loci of $\mc{C}(n)$ for small odd values of $n$; just as in Figure 1, the validity of Conjectures~\ref{topological_profile_of_real_locus} and \ref{numerology_of_key_polynomial_roots} in these cases is graphically apparent.
 
 \begin{figure}[!htb]
 \includegraphics{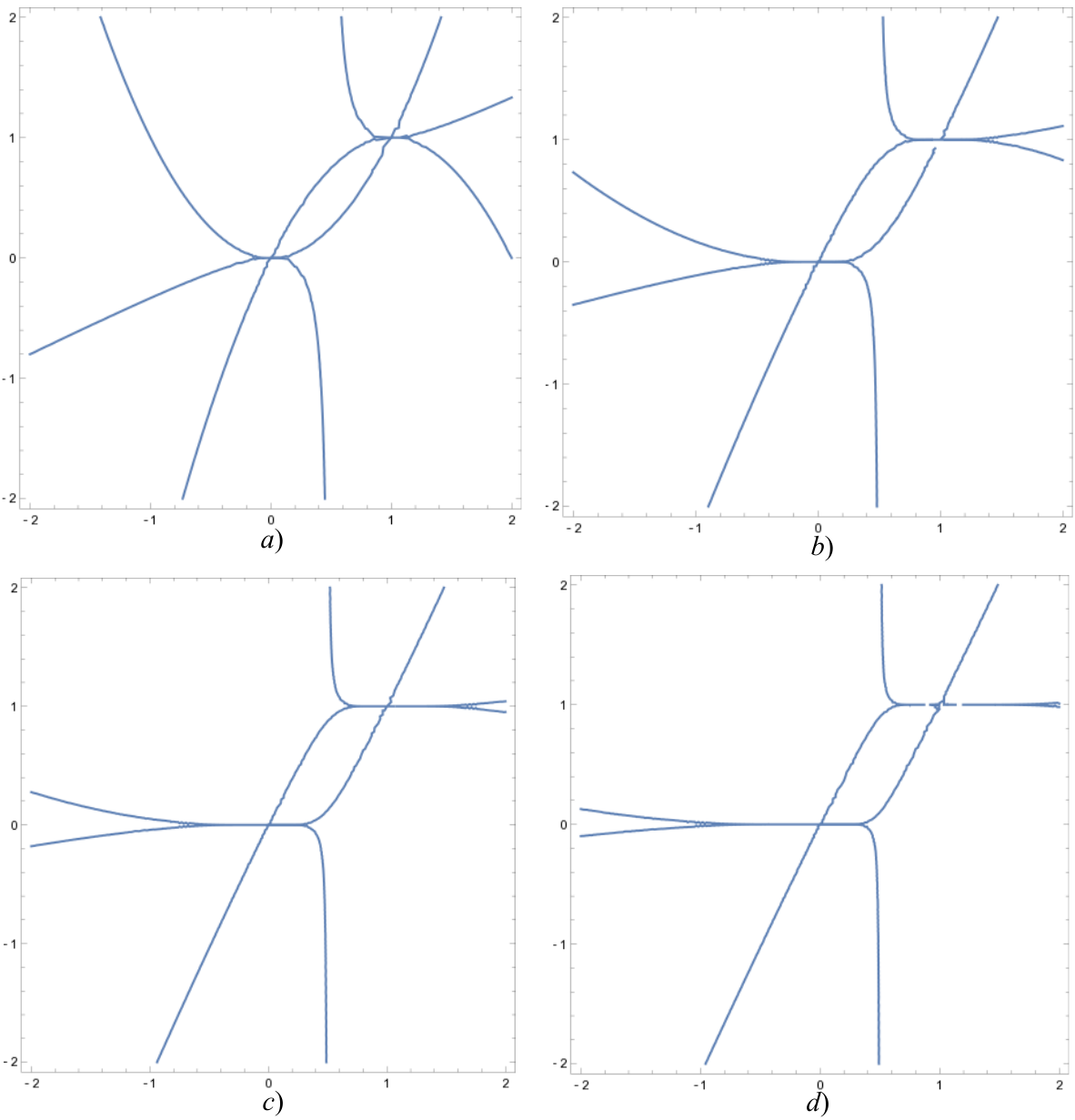}
 \caption{Real loci of a) $P_3=0$, b) $P_5=0$, c) $P_7=0$, and d) $P_9=0$.}
 \end{figure}

\section{From incomplete series on an elliptic curve to complete series on a hyperelliptic curve}
 
 Conjectures~\ref{topological_profile_of_real_locus} and \ref{numerology_of_key_polynomial_roots}, assuming they hold, may be leveraged to construct complete real linear series on real hyperelliptic curves of genus $g \geq 1$ with many real inflection points.

\begin{thm}
Assume that Conjecture~\ref{numerology_of_key_polynomial_roots} holds. There exists a maximally-real hyperelliptic curve $X$ with affine equation $y^2=f$ admits a complete real linear series $|\mc{L}_\mathbb{R}|=|\mc{L}_\mathbb{R}(kD)|$ with the distinguished basis $\mc{F}_g$ as in \eqref{distinguished_basis} whose total real inflectionary degree is
\[
\om_{\mb{R}}(k,g)= 2g(g+1)+2(k-g)(g-1)+ 2g(1+ g \hspace{-7pt}\mod 2)(k-g).
\]
Here $f=f(x)$ is a polynomial of degree $2g-1$ ramified over $\infty \in \mb{P}^1$, and $D$ is the divisor represented by the pullback of $\infty$ to $X$.
\end{thm}

\begin{proof}
Follows immediately from Theorems 5.8 and 6.1 of \cite{BCG}.
\end{proof}

\end{document}